\documentclass[11pt,oneside,a4paper]{amsart}
\usepackage{amsmath,amssymb,amsthm}
\usepackage{mathrsfs,scalefnt,tikz} 
\usepackage{tikz,tikz-cd}
\usepackage{ytableau}
\usepackage[pagebackref]{hyperref} 
\hypersetup{colorlinks,linkcolor=blue,urlcolor=blue,citecolor=blue}
\usepackage[alphabetic,abbrev,nobysame]{amsrefs} 

\newtheorem{thm}{Theorem}[section]
\newtheorem*{thm*}{Theorem}
\newtheorem{lem}[thm]{Lemma}
\newtheorem*{lem*}{Lemma}
\newtheorem{prop}[thm]{Proposition}
\newtheorem*{prop*}{Proposition}
\newtheorem{cor}[thm]{Corollary}
\newtheorem*{cor*}{Corollary}

\newtheorem*{conj*}{Conjecture}

\theoremstyle{definition}
\newtheorem{defn}[thm]{Definition}
\newtheorem*{defn*}{Definition}
\newtheorem{example}[thm]{Example}
\newtheorem*{example*}{Example}

\newtheorem*{prob*}{Problem}

\newtheorem*{probs*}{Problems}
\newtheorem{rmk}[thm]{Remark}
\newtheorem*{rmk*}{Remark}

\newtheorem*{quest*}{Question}

\newcommand{\C}{\mathcal{C}} 
\newcommand{\Z}{\mathbb{Z}} 
\newcommand{\End}{\operatorname{End}} 
\newcommand{\V}{\mathbf{V}} 
\newcommand{\vv}{\mathbf{v}} 
\newcommand{\Sym}{\mathfrak{S}} 
\newcommand{\Ptn}{\mathcal{P}} 
\newcommand{\im}{\operatorname{im}} 
\newcommand{\B}{\mathcal{B}} 
\newcommand{\opp}{{\operatorname{op}}} 
\newcommand{\ui}{\underline{i}}
\newcommand{\uj}{\underline{j}}
\newcommand{\up}{\underline{p}}
\newcommand{\uq}{\underline{q}}
\newcommand{\GL}{\operatorname{GL}} 

\newcommand{\LLIS}{\operatorname{LLIS}} 
\newcommand{\LLDS}{\operatorname{LLDS}} 

\renewcommand{\labelenumi}{(\alph{enumi})}
\allowdisplaybreaks

\title[New applications of increasing and decreasing subsequences]%
      {Canonical bases and new applications \\
        of increasing and decreasing subsequences \\
        to invariant theory}

  \author{Chris Bowman}
  \address{Department of Mathematics, 
    University of York, Heslington, York, YO10 5DD, UK}
  \email{Chris.Bowman-Scargill@york.ac.uk}

  \author{Stephen Doty}
  \address{Department of Mathematics and Statistics, Loyola University
    Chicago, Chicago, IL 60660 USA}
  \email{doty@math.luc.edu}

  \author{Stuart Martin}
  \address{DPMMS, Centre for Mathematical Sciences, Wilberforce Road,
  Cambridge, CB3 0WB, UK}
  \email{sm137@cam.ac.uk}

\begin{document}
\begin{abstract}\noindent
In 2012 Raghavan, Samuel, and Subrahmanyam showed that the
Kazhdan--Lusztig basis for the Iwahori--Hecke algebra in type A
provides a ``canonical'' basis for the centraliser algebra of the
Schur algebra acting on tensor space.  In 2022 the second author found
a similar result for the centraliser of the partition algebra acting
on the same tensor space.  Each basis is indexed by permutations.  We
exploit these bases to show that the linear decomposition of an
arbitrary invariant (in either centraliser algebra) depends integrally
on its entries, and describe combinatorial rules that pick out minimal
sets of such entries.
\end{abstract}
\maketitle

\vspace{-0.3cm}
\section*{\bf Introduction}
\noindent 
We study centraliser algebras for the actions on tensor space of two
finite dimensional algebras: the Schur algebra and partition algebra.
Thanks to two instances of Schur--Weyl duality, these centraliser
algebras are isomorphic to quotients of group algebras of symmetric
groups. 
 This paper focuses on the problems:
\begin{enumerate}
\item[(Q1)] Find a ``canonical'' basis for the centraliser algebras.
\item[(Q2)] Given such a basis, calculate the coefficients expressing a
  given invariant as a linear combination of the basic ones.
\end{enumerate}
Previous results \cites{RSS,D} solve (Q1) for both centraliser
algebras (see Theorem \ref{t:2}).  The main results of this paper are
given in Section \ref{s:2}, where we provide a combinatorial algorithm
for computing the coefficients of a given invariant expressed as a
linear combination of the canonical basis, solving (Q2) for both
algebras.

Thus our work has connections to: invariant theory, symmetric groups,
Schur--Weyl dualities, combinatorics (including the RSK
correspondence), and computer science.  Our motivation comes from the
connections to symmetric groups and invariant theory (Schur--Weyl
duality), and we hope that the existence of the algorithm will foster
subsequent applications to these areas.

In the case of the Schur algebra invariants, if $\Bbbk$ is any
infinite field then the first of the aforementioned algebras is
\[
\End_{S(n,r)}(\V^{\otimes r}) = \End_{\GL_n(\Bbbk)}(\V^{\otimes r})
\]
where $\V$ is the natural module for $\GL_n(\Bbbk)$ and $S(n,r)$ is
the Schur algebra.  Thus, over an infinite field our results are
applicable to the classical problem of decomposing $\GL_n$-invariants.
Furthermore, the above centraliser algebras are related to
(generalised) Temperley--Lieb algebras studied in
\cites{Fan,Fan:97,Stembridge,Haerterich}.

We believe that the use of the word ``canonical'' to describe the
bases in Theorem~\ref{t:2} is justified by the origins of the bases in
the papers \cites{RSS,D}, where it was shown that they arise as
leading terms of certain Kazhdan--Lusztig basis \cite{KL} elements
(for the Iwahori--Hecke algebra of type A) evaluated at $q=1$.
The canonical bases of  \cites{RSS,D} consist of the (image of) permutations satisfying
a combinatorial condition formulated in terms of the length of their
longest increasing or decreasing subsequences.

The paper is organized as follows. In Section~\ref{s:1} we summarise
the aforementioned previous results that underlie this work. The main
results are formulated in Section~\ref{s:2} and proved in
Sections~\ref{s:3} and~\ref{s:4}. Finally, some applications are
considered in Section~\ref{s:5}, one of which answers a question of
Rouquier.

\section{\bf Summary of prior literature}\label{s:1}\noindent
Let $\Bbbk$ be a unital commutative ring, and fix a basis $\{\vv_1,
\dots, \vv_n\}$ for a free $\Bbbk$-module $\V$ of rank $n$.  Let
$\V^{\otimes r} = \V \otimes \cdots \otimes \V$ ($r$ factors) be the
$r$th tensor power of $\V$.  The parameters $n$, $r$ are
independent. \emph{Throughout the rest of the paper, $n$ and $r$
always refer to the rank of $\V$ and the number of tensor factors in
$\V^{\otimes r}$, respectively.}

There are natural right and left commuting actions of the two
symmetric groups $\Sym_r$, $W_n$ on $\V^{\otimes r}$, given
respectively by:
\begin{align}
  (\vv_{j_1} \otimes \vv_{j_2} \otimes \cdots \otimes \vv_{j_r})\cdot
  \sigma &= \vv_{j_{\sigma(1)}}\otimes \vv_{j_{\sigma(2)}} \otimes
  \cdots \otimes \vv_{j_{\sigma(r)}} \label{e:1} \\
  w \cdot (\vv_{j_1} \otimes \vv_{j_2} \otimes \cdots \otimes \vv_{j_r})
  & = \vv_{w(j_1)}\otimes \vv_{w(j_2)} \otimes \cdots \otimes \vv_{w(j_r)}
  \label{e:2}
\end{align}
for any $\sigma \in \Sym_r$, $w \in W_n$.  The action of $\Sym_r$ is
by place-permutation and the action of $W_n$ is the diagonal extension
of the natural permutation action of $W_n$ on the given basis. We use
different symbols for the two symmetric groups to avoid the
possibility of confusing one for the other. Even when $n=r$ they act
differently on the tensor space $\V^{\otimes r}$. Note that $W_n$ may
be identified with the Weyl group of $\GL(\V) \cong \GL_n(\Bbbk)$ and
the action of $W_n$ is the one induced by the natural diagonal action
of $\GL(\V)$.

In Theorems \ref{t:1} and \ref{t:2} below we summarise the recent
results underlying the question addressed by this paper. Let
\begin{align*}
\Phi = \Phi_{n,r}: \Bbbk[\Sym_r]^\opp \to \End_\Bbbk(\V^{\otimes
  r}),\quad \Psi = \Psi_{n,r}: \Bbbk[W_n] \to \End_\Bbbk(\V^{\otimes
  r})
\end{align*}
be the representations afforded by the respective actions \eqref{e:1},
\eqref{e:2}. Each representation $\Phi = \Phi_{n,r}$ and $\Psi =
\Psi_{n,r}$ depends on both $n$ and $r$.

\begin{thm}\label{t:1}
Let $\Bbbk$ be a unital commutative ring. Let $n \ge 1$ and $r\ge 0$
be integers. Then
\[
\text{\upshape{(i)} $\im(\Phi_{n,r}) = \End_{S(n,r)}(\V^{\otimes r})$, and
\upshape{(ii)} $\im(\Psi_{n,r}) = \End_{\Ptn_r(n)}(\V^{\otimes r})$.}
\]
\end{thm}

\begin{proof}[Proof sketch, with historical context]
We will now outline how these results were proved, including pointers
to the literature to aid the reader.  As an $(\GL(\V),
\Bbbk[\Sym_r]^\opp)$-bimodule, $\V^{\otimes r}$ satisfies Schur--Weyl
duality, in the sense that the image of each action generates the
centraliser for the other. When $\Bbbk=\mathbb{C}$ this is a classical
result of Schur, and over infinite fields it is due to \cites{Green,
  DP}.  The result was extended to sufficiently large finite fields in
\cite{BD}. The paper \cite{Cruz} showed that it fails in general when
$\Bbbk$ is a commutative ring.

However, the Schur algebra $S(n,r)$ is defined over $\Z$ and acts on
$\V^{\otimes r}$ as described in \cite{Green}*{\S2.6}. That action
commutes with the place-permutation action of $\Sym_r$, so
$\V^{\otimes r}$ is an $(S(n,r), \Bbbk[\Sym_r]^\opp)$-bimodule. By
extending an argument in \cite{BD}, \cite{Cruz}*{Thm.~3.4} proved
equality (i), for any unital commutative ring $\Bbbk$. Incidentally,
this implies that Schur--Weyl duality holds for the $(S(n,r),
\Bbbk[\Sym_r]^\opp)$-bimodule structure on $\V^{\otimes r}$, in the
same generality, thanks to the well known equality $S(n,r)
= \End_{\Sym_r}(\V^{\otimes r})$, which goes all the way back to
Schur. This equality holds for any unital commutative ring $\Bbbk$
(see the proof of \cite{Green}*{(2.6c)}).

The partition algebra $\Ptn_r(n)$ acts on $\V^{\otimes r}$ as
described in \cite{HR}*{eq.~(3.2)}.  As a
$(\Bbbk[W_n],\Ptn_r(n))$-bimodule, $\V^{\otimes r}$ also satisfies
Schur--Weyl duality over any unital commutative ring $\Bbbk$, by the
main result of \cite{BDM} (see also \cite{Donkin}*{\S6}), so in
particular equality (ii) holds.
\end{proof}

\begin{thm}\label{t:2}
Let $\LLIS$ (resp., $\LLDS$) denote the length of a longest increasing
(resp., decreasing) subsequence of a permutation.  Let
\begin{align*}
\B &= \B_{n,r} = \{\sigma \in \Sym_r \mid \LLDS(\sigma) \le n\}, \\
\C &= \C_{n,r} = \{w \in W_n \mid \LLIS(w) \ge n-r\}.
\end{align*}
Then we have
\begin{enumerate}
 \item [(i)] $\{\Phi(\sigma)\}_{\sigma \in \B_{n,r}}$ is a
   $\Bbbk$-basis of $\End_{S(n,r)}(\V^{\otimes r})$.
 \item [(ii)] $\{\Psi(w)\}_{w \in \C_{n,r}}$ is a $\Bbbk$-basis of
   $\End_{\Ptn_r(n)}(\V^{\otimes r})$.
\end{enumerate}
\end{thm}

See \cite{RSS}*{Thm.~1} for the statement for $\Phi$, and 
\cite{D}*{Thm.~1} for the statement for $\Psi$.

\begin{rmk}
Although the statements of both \cite{RSS}*{Thm.~1} and
\cite{D}*{Thm.~1} are statements about representations of symmetric
groups, the proof in both cases takes place in the corresponding
Iwahori--Hecke algebra and relies on deeper results of Geck's paper
\cite{Geck} relating its Kazhdan--Lusztig bases to its Murphy bases.
The connection to increasing and decreasing subsequences comes from
the Robinson--Schensted correspondence \cite{Schensted}.  Furthermore,
the proof of \cite{D}*{Thm.~1} follows from previous joint work
\cite{BDM:kernel} of the authors and results of Donkin
\cite{Donkin}*{\S8} that obtains a $q$-analogue of one of the results
in \cite{BDM:kernel}.
\end{rmk}

Combining Theorems \ref{t:1} and \ref{t:2} gives the following
immediate consequence. Write $\bar{x} = x+J$ for the image of an
element $x$ under the canonical quotient map $A \to A/J$, where $A$ is
an algebra and $J$ an ideal.

\begin{cor}\label{c:1}
Let $\Bbbk$ be a unital commutative ring.
\begin{enumerate}
 \item [(i)] $\overline{\B}_{n,r} = \{\bar{\sigma} \mid \sigma \in
   \B_{n,r}\}$ is a $\Bbbk$-basis of $\Bbbk[\Sym_r]^\opp/\ker(\Phi_{n,r})$.
 \item [(ii)] $\overline{\C}_{n,r} = \{\bar{w} \mid w \in \C_{n,r}\}$
   is a $\Bbbk$-basis of $\Bbbk[W_n]/\ker(\Psi_{n,r})$.
\end{enumerate}
\end{cor}

Notice that $\B_{1,r} = \{id\}$ consists of only the identity
permutation, $\B_{r-1,r} = \Sym_r \setminus \{\sigma_0\}$, where
$\sigma_0$ is the element of longest Coxeter length, and we have
inclusions
\[
\B_{1,r} \subset \B_{2,r} \subset \cdots \subset \B_{r,r} = \Sym_r .
\]
(Clearly, $\B_{n,r} = \Sym_r$ stabilises if $n\ge r$).  Furthermore,
in case (ii), $\C_{n,0} = \{id\}$, $\C_{n,n-2} = W_n \setminus
\{w_0\}$, where $w_0$ is the element of longest Coxeter length, and we
have inclusions
\[
\C_{n,0} \subset \C_{n,1} \subset \cdots \subset \C_{n,n-1} = W_n .
\]
(In this case, $\C_{n,r} = W_n$ stabilises if $r \ge n-1$).

\begin{example}\label{ex:bases}
In this and subsequent examples, we write permutations in the one-line
notation, omitting punctuation (commas) between entries. We have:
\begin{gather*}
\B_{1,3} = \{(123)\}, \quad \B_{2,3} = \Sym_{3} \setminus \{(321)\},
\quad \text{and} \quad \B_{3,3} = \Sym_3.
\\
\C_{3,0} = \{(123)\},\quad \C_{3,1} =
W_3 \setminus \{(321)\},\quad  \text{and} \quad \C_{3,2} = W_3.
\\
\B_{1,4} = \{(1234)\},\\
\B_{2,4} = \{(1234),(1243),(1324),(2134),(1342),(1423),(2143), \\ (2314),
  (3124),(2341),(2413),(3142),(4123),(3412)\}, \\ \B_{3,4} =
\Sym_4 \setminus \{(4321)\},\quad \text{and}\quad \B_{4,4} = \Sym_4.
\\
\C_{4,0} = \{(1234)\},\\
\C_{4,1} = \{(1234),(1243),(1324),(2134),(1342), \\
  (1423),(2314),(3124),(2341),(4123)\}, \\ \C_{4,2} =
W_4 \setminus \{(4321)\},\quad \text{and}\quad \C_{4,3} = W_4.
\end{gather*}
It is known \cite{D}*{Prop.~4} that $\C_{n,1}$ is the set of
``consecutive cycles'' in $W_n$, for any $n$. (A permutation is a
consecutive cycle if it or its inverse is a cycle of the form $i
\mapsto i+1 \mapsto \cdots \mapsto i+k-1 \mapsto i$ for some $i,k$.)
\end{example}

The kernels of the maps $\Phi_{n,r}$ and $\Psi_{n,r}$ are well
understood; see \cite{DP}*{Thm.~4.2} for the former and
\cite{BDM:kernel} for the latter. In particular, the maps $\Phi_{n,r}$
and $\Psi_{n,r}$ are injective (the corresponding symmetric groups act
faithfully) if and only if $n \ge r$ and $r \ge n-1$, respectively.
Cellular bases for each kernel are known; see \cite{Haerterich} for
$\Phi$ and \cite{BDM:kernel} for $\Psi$. The kernels are cell ideals,
so each quotient is a cellular algebra, in the sense of \cite{GL}. 

\begin{cor}\label{c:2}
Let $\Bbbk$ be a commutative ring with unit. Define $\mathbf{B}(n,r)
:= \Bbbk[\Sym_r]^\opp/\ker(\Phi_{n,r})$ and $\mathbf{C}(n,r) :=
\Bbbk[W_n]/\ker(\Psi_{n,r})$. These are cellular algebras, and we have
the following.
\begin{enumerate}
\item[(i)] Fixing $r$ and letting $n$ vary from $1$ to $r$ produces
  a sequence of (surjective) quotient maps
  \[
  \mathbf{B}(1,r) \leftarrow \mathbf{B}(2,r) \leftarrow \cdots
  \leftarrow \mathbf{B}(r,r) = \Bbbk[\Sym_r]^\opp
  \]
  such that each $\sigma$ in $\Sym_r$ maps to the corresponding coset
  in each quotient.

\item[(ii)] Similarly, fixing $n$ and letting $r$ vary from $0$ to
  $n-1$ produces a sequence of quotient maps
  \[
  \mathbf{C}(n,0) \leftarrow \mathbf{C}(n,1) \leftarrow \cdots
  \leftarrow \mathbf{C}(n,n-1) = \Bbbk[W_n]
  \]
  such that each $w$ in $W_n$ maps to the corresponding coset
  in each quotient.
\end{enumerate}
\end{cor}

\begin{proof}
To prove (i), it suffices to observe that there is a descending series
of two-sided ideals
\[
\ker(\Phi_{1,r}) \supset \ker(\Phi_{2,r}) \supset \cdots \supset
\ker(\Phi_{r,r}) = (0).
\]
To prove (ii), we observe that there is a descending series
of two-sided ideals
\[
\ker(\Psi_{n,0}) \supset \ker(\Psi_{n,1}) \supset \cdots \supset
\ker(\Psi_{n,n-1}) = (0).
\]
These observations follow from the descriptions of the kernels
given in the references cited in the remarks preceding the proof.
\end{proof}

In light of Corollary~\ref{c:2}, it is natural to ask whether or not
the bases in Corollary~\ref{c:1} are cellular bases.
Example~\ref{ex:bases} shows that the answer is negative in
general. To see this, recall that K\"{o}nig and Xi \cite{KX} showed
that an algebra with a finite basis is cellular if and only if it has
a certain chain of ideals known as a cell chain. In particular, if an
algebra with an anti-involution is cellular, it would have to contain
a cell ideal or be isomorphic to a matrix algebra.  But the ideal
generated by any of the elements of $\bar{\C}_{3,1}$ is the entire
algebra $\mathbf{C}(3,1)$, so it cannot have a proper cell ideal
containing an element of $\bar{\C}_{3,1}$.  Furthermore, $|\C_{3,1}| =
|\bar{\C}_{3,1}| = 5$, so $\mathbf{C}(3,1)$ cannot be isomorphic to a
matrix algebra, either. The same argument shows that $\bar{\B}_{2,3}$
is not a cellular basis of $\mathbf{B}(2,3)$. The argument easily
extends to other cases.

\begin{rmk}
When a permutation is written in the one-line notation, reversing its
order interchanges the notions of increasing and decreasing
subsequences. This order reversal map is an involution on the ambient
symmetric group \cite{Schensted}. It is easy to see that order
reversal is the map given by right multiplication\footnote{This
assumes that we compose maps from right-to-left, so that $(f \circ
g)(x) = f(g(x))$ If we compose maps from left-to-right then we should
replace right multiplication by left multiplication.} by the (unique)
element of longest Coxeter length. The kernels of $\Phi_{n,r}$ and
$\Psi_{n,r}$ are invariant under such multiplication. 
It follows that order reversal induces an involution on both
$\Bbbk[\Sym_r]^\opp/\ker(\Phi_{n,r})$ and
$\Bbbk[W_n]/\ker(\Psi_{n,r})$. So we can interchange the notions of
$\LLIS$, $\LLDS$ in the definition of the sets $\B_{n,r}$ and
$\C_{n,r}$ and all the results of this section remain true. This
implies that all the results of this paper have a ``dual'' form in
which such an interchange has been effected.
\end{rmk}

\section{\bf The main results}\label{s:2}\noindent
In this section we fix $n$ and $r$ and write $\Phi = \Phi_{n,r}$,
$\Psi = \Psi_{n,r}$, $\B=\B_{n,r}$, $\C = \C_{n,r}$.  The problem
considered in this paper is the following: given any invariant $X$ in
$\End_{S(n,r)}(\V^{\otimes r})$, $\End_{\Ptn_r(n)}(\V^{\otimes r})$,
find the (unique) scalars $c(\sigma)$, $c(w)$ in $\Bbbk$ such that
\begin{equation}\label{e:problem}
  \sum_{\sigma \in \B} c(\sigma)\Phi(\sigma) = X, \qquad \sum_{w \in
    \C} c(w) \Psi(w) = X 
\end{equation}
respectively.  In other words, we are considering (Q2) for the above
centraliser algebras, with respect to the bases of Theorem~\ref{t:2}.

We need some notation in order to formulate our main result. Set $[n]
:= \{1,\dots, n\}$. Given any multi-index $\uj = (j_1, j_2, \dots,
j_r)$ in $[n]^r$, set
\[
\vv_{\uj} := \vv_{j_1} \otimes \vv_{j_2} \otimes \cdots \otimes \vv_{j_r}.
\]
The set $\{\vv_{\uj} \mid \uj \in [n]^r\}$ is a $\Bbbk$-basis of
$\V^{\otimes r}$.  Let $(\uj,\sigma)\mapsto \uj \sigma$,
$(w,\uj)\mapsto w\uj$ denote the actions of $\Sym_r$, $W_n$ on $[n]^r$
such that
\[
\uj \sigma = (j_{\sigma(1)}, \dots, j_{\sigma(r-1)}, j_{\sigma(r)}), \quad
w \uj = (w(j_1), \dots, w(j_{r-1}), w(j_r)).
\]
This notation is compatible with equations \eqref{e:1} and \eqref{e:2}
in the sense that $\vv_{\uj} \cdot \sigma = \vv_{\uj \sigma}$ and $w
\cdot \vv_{\uj} = \vv_{w \uj}$.  Taking matrix coordinates in equation
\eqref{e:problem} with respect to this basis, we obtain the
(overdetermined, in general) linear systems
\begin{equation}\label{e:4}
  \sum_{\sigma \in \B} c(\sigma)\Phi(\sigma)_{\ui,\uj} = X_{\ui,\uj},
  \qquad \sum_{w \in \C} c(w) \Psi(w)_{\ui,\uj} = X_{\ui,\uj} 
\end{equation}
as $(\ui,\uj)$ vary over all pairs of multi-indices in $[n]^r \times
[n]^r$. By Theorem~\ref{t:2}, the coefficents in each system are
uniquely determined. What is needed is a method of picking out a
minimal linear subsystem producing the unique solution in each case.
Our contribution is to find combinatorial maps that do this job.  Here
is our main result.

\begin{thm}\label{t:3}
Injective maps $\varphi: \B \to [n]^r \times [n]^r$ and
$\psi: \C \to [n]^r \times [n]^r$ exist such that:
\begin{enumerate}
\item[(i)] If $\varphi(\sigma) = (\ui,\uj)$ then $\sigma$ is the
  unique element of $\B$ of minimal Coxeter length satisfying
  $\ui\sigma = \uj$.
\item[(ii)] If $\psi(w) = (\ui,\uj)$ then $w$ is the unique
  element of $\C$ of minimal Coxeter length satisfying $w \uj = \ui$.
\end{enumerate}
\end{thm}

Part (i) is proved in Proposition~\ref{p:2} and part (ii) in
Proposition~\ref{p:3} below.  Given any $\sigma'$ in $\B$ and any $w'$
in $\C$ we write
\[
\varphi(\sigma') = (\ui(\sigma'), \uj(\sigma')) \quad\text{and}\quad
\psi(w') = (\ui(w'), \uj(w')).
\]
For $X = [X_{\ui,\uj}]_{\ui,\uj \in [n]^r}$ in either
$\End_{S(n,r)}(\V^{\otimes r})$ or $\End_{\Ptn_r(n)}(\V^{\otimes r})$,
consider the corresponding linear subsystems
\begin{equation}\label{e:5}
\begin{aligned}
\textstyle \sum_{\sigma \in \B} c(\sigma)\Phi(\sigma)_{\ui(\sigma'),
  \uj(\sigma')} &= X_{\ui(\sigma'), \uj(\sigma')}, \\ \textstyle
\sum_{w \in \C} c(w) \Psi(w)_{\ui(w'), \uj(w')} &= X_{\ui(w'),
  \uj(w')}
\end{aligned}
\end{equation}
of the systems in \eqref{e:4}, as $\sigma'$ varies over $\B$ and $w'$
varies over $\C$. By Theorem~\ref{t:3}, these subsystems have
unitriangular coefficient matrices
\[
A_\Phi = \big[\Phi(\sigma)_{\ui(\sigma'),\uj(\sigma')} \big]_{\sigma,\sigma' \in \B}
\; , \qquad A_\Psi = \big[\Psi(w)_{\ui(w'),\uj(w')} \big]_{w,w' \in \C}
\]
if their rows and columns are ordered by any total ordering of $\B$,
$\C$ compatible with the Coxeter length function.  We observe that the
matrices $A_\Phi$ and $A_\Psi$ are $(0,1)$-matrices. To see this,
notice that for any $\sigma \in \Sym_r$ or $w \in W_n$, we have
\begin{equation}
  \Phi(\sigma)_{\ui,\uj} = \delta_{i_{\sigma(1)},j_1} \cdots
  \delta_{i_{\sigma(r)},j_r}, \qquad
  \Psi(w)_{\ui,\uj} = \delta_{i_1,w(j_1)} \cdots \delta_{i_r,w(j_r)}
\end{equation}
for any $\ui,\uj$ in $[n]^r$, where $\delta_{ij}$ is the usual
Kronecker delta symbol. It follows that the entries of $A_\Phi^{-1}$
and $A_\Psi^{-1}$ are integers. In summary, we have the following.

\begin{cor}
The $c(\sigma)$ or $c(w)$ in \eqref{e:problem} are respectively
expressible as integral linear combinations of the
$X_{\ui(\sigma'),\uj(\sigma')}$ or $X_{\ui(w'),\uj(w')}$, as $\sigma'$
varies over $\B$ and $w'$ over $\C$.
\end{cor}

Theorem~\ref{t:3} can be said to ``solve'' the problem of computing
the $c(\sigma)$, $c(w)$ in equations \eqref{e:problem}, by reducing
the problem to the calculation of the inverse of a unitriangular
$(0,1)$-matrix. Such a calculation is relatively easy for a digital
computer, as no multiplications or divisions are required.

\begin{rmk}
In all the examples known to the authors, it turns out that the
entries of $A_\Phi^{-1}$ and $A_\Psi^{-1}$ all belong to the set
$\{0,1,-1\}$. We do not know whether or not such a statement holds in
general.
\end{rmk}

\section{\bf The map $\varphi$}\label{s:3}\noindent
We continue to assume that $\Bbbk$ is a unital commutative ring. We
write $\B = \B_{n,r}$, $A = A_\Phi$ in this section.  We now define a
map $\varphi': \Sym_r \to [r]^r \times [r]^r$ which, for any positive
integer $n$, restricts to a map $\varphi: \B \to [n]^r \times [n]^r$
having the properties needed to complete the proof of Theorem
\ref{t:3}(i). The map $\varphi'$ depends on $r$ and is independent of
$n$.

\begin{defn}\label{d:initial}
The \emph{initial increasing subsequence} $M = (y_{1}, y_{i_2}, \dots,
y_{i_\ell})$ of a given finite sequence $y = (y_i)$ of real numbers is
defined as follows:
\begin{enumerate}
\item Set $i_1 = 1$ and $M_1 = (y_{i_1}) = (y_1)$.
\item For any $a \ge 1$, assuming that $(i_1, \dots, i_a)$ and
  $M_a=(y_{i_1}, \dots, y_{i_a})$ have been selected so far, let
  $i_{a+1}$ be the unique smallest index $i$ for which $i_{a+1} > i_a$ and
  $y_{i_{a+1}} > y_{i_a}$, if possible, and set $M_{a+1} = (y_{i_1},
  \dots, y_{i_a}, y_{i_{a+1}})$.
\item The process terminates when there is no such element in the
  given sequence. Let $M = M_{\mathrm{max}}$, the final subsequence at
  time of termination.
\end{enumerate}
\end{defn}

\begin{example}\label{ex:initial}
The initial increasing subsequence of $y = (5,2,1,3,7,4,6,9)$ is
$(5,7,9)$. Note that this is not a longest increasing subsequence.
\end{example}

Now let $x=(x_1,x_2,\dots, x_r)$, $y=(y_1,y_2,\dots, y_r)$ be given
finite sequences of the same length, each consisting of
\emph{distinct} real numbers (or elements of any linearly ordered
set), such that the first sequence is increasing: $x_1<x_2<\cdots
<x_r$.  We may display the pair of sequences in a two-line notation by
writing
\[
(x/y) = 
\begin{pmatrix}
  x_1&x_2&\cdots&x_r\\
  y_1&y_2&\cdots&y_r
\end{pmatrix} .
\]
The standard way of computing the length of the longest decreasing
subsequence of $y$ is the tableau-based Robinson--Schensted algorithm
\cites{Schensted,Fulton}. But there are other (less efficient) methods
of computing that statistic.  The map $\varphi$ is based on one such
method. As we show in Proposition~\ref{p:1}, it computes the length of
a longest decreasing subsequence.

Given $(x/y)$ with $x$ increasing, let $(y_1, y_{i_2}, \dots,
y_{i_\ell})$ be the initial increasing subsequence of $y$.  Let
$(x_{1}, x_{i_2}, \dots, x_{i_\ell})$ be the corresponding
(necessarily increasing) subsequence of $x$. Write
\[
\pi_1(x/y) =
\begin{pmatrix}
  x_{1}&x_{i_2}&\cdots&x_{i_\ell} \\
  y_{1}&y_{i_2}&\cdots&y_{i_\ell}
\end{pmatrix} 
\]
for the resulting pair of increasing subsequences of $(x/y)$. The
\emph{complement} $(x'/y') = (x/y) \setminus \pi_1(x/y)$ is obtained
by excising the columns of $\pi_1(x/y)$ from $(x/y)$ and reindexing.
If $y$ is increasing then the complement is empty and
$\pi_1(x/y)=(x/y)$. Otherwise, the complement $(x'/y')$ is a new pair
of sequences satisfying the same conditions as $(x/y)$. So we may
apply $\pi_1$ again to the pair $(x'/y')$, obtaining $\pi_2(x/y) =
\pi_1(x'/y')$, a second pair of increasing subsequences of
$(x/y)$. Continuing in this manner, we obtain the \emph{canonical
factorisation} of the given pair $(x/y)$, as the configuration
\[
\pi(x/y) := \big( \pi_1(x/y) \mid \pi_2(x/y) \mid \cdots \mid
\pi_k(x/y) \big)
\]
consisting of pairs of increasing subsequences of $(x/y)$. The integer $k$
is the \emph{length} of the canonical factorisation.
 (See \cite{greene} for a similar construction.)

\begin{example}\label{ex:1}
If $(x/y)$ is the pair of sequences given by 
\[
(x/y) = \begin{pmatrix}
  1&2&3&4&5&6&7&8&9 \\
  4&2&1&7&3&8&9&6&5
\end{pmatrix}
\]
then we have
\[
\pi_1(x/y) = \begin{pmatrix}
  1&4&6&7 \\
  4&7&8&9
\end{pmatrix},\quad
\pi_2(x/y) = \begin{pmatrix}
  2&5&8 \\
  2&3&6
\end{pmatrix}, \quad
\pi_3(x/y) = \begin{pmatrix}
  3&9 \\
  1&5
\end{pmatrix} .
\]
We write this succinctly in the form
\[
\pi(x/y) = \left(
\begin{array}{cccc|ccc|cc}
  1&4&6&7&2&5&8&3&9\\
  4&7&8&9&2&3&6&1&5
\end{array}
\right) 
\]
in which vertical lines are used to separate the various $\pi_i(x/y)$
in their natural order $\pi_1(x/y), \dots, \pi_k(x/y)$.
\end{example}

We apply the above observations to permutations. By writing $\sigma
\in \Sym_r$ in the usual two-line notation, we obtain a pair of
sequences $(x/y)$ as above (with the first-line sequence $x =
(1,2,\dots,r)$ increasing). For instance, the pair of sequences
$(x/y)$ written in Example \ref{ex:1} is the two-line notation for the
permutation $\sigma = (421738965)$ in the one-line notation. Write
\[
\pi(\sigma) = \big( \pi_1(\sigma) \mid \pi_2(\sigma) \mid \cdots \mid
\pi_k(\sigma) \big)
\]
with the obvious interpretation, namely, that $\pi_i(\sigma) =
\pi_i(x/y)$, for all $i= 1,\dots,k$, where $k$ is the length of the
canonical factorisation.

We now define the promised map $\varphi': \Sym_r \to [r]^r \times
[r]^r$, which will soon be used to define the map $\varphi$.  Given
$\sigma \in \Sym_r$, let $k$ be the length of its canonical
factorisation, and set $[k] := \{1,\dots, k\}$. Let $i = (i_1, i_2,
\dots, i_r)$, $j = (j_1,j_2,\dots,j_r)$ be the sequences in $[k]^r$
defined by the conditions
\begin{align*}
i_\alpha = c &\iff \alpha \text{ appears in the first line of } \pi_c(\sigma) \\
j_\alpha = c &\iff \alpha \text{ appears in the second line of } \pi_c(\sigma).
\end{align*}
In other words, we colour each number in $\pi_c(\sigma)$ by the colour
$c$, for each $c$ in $[k]$, and let $\ui$, $\uj$ be the corresponding
sequences of colours in the first and second line of $\pi(\sigma)$,
resp. Let $\varphi'(\sigma) = (\ui,\uj)$. Evidently, $\varphi'(\sigma)$
is in $[r]^r \times [r]^r$.

If $\sigma_0$ is the element of $\Sym_r$ of longest Coxeter length,
then
\[
\varphi'(\sigma_0) = ((1,2,\dots,r), (r,r-1,\dots,1)),
\]
which shows that the image of $\varphi'$ is not contained in any $[n]^r
\times [n]^r$, for $n < r$. At the opposite extreme, if $id$ is the
identity permutation then
\[
\varphi'(id) = ((1,1,\dots,1), (1,1,\dots,1)) .
\]
For the permutation $\sigma = (421738965)$ of Example~\ref{ex:1}, we
have
\[
\varphi'(\sigma) = ((1,2,3,1,2,1,1,2,3), (3,2,2,1,3,2,1,1,1)).
\]

The following lemma is similar to Lemma 5 in \cite{Schensted}.

\begin{lem}\label{lem:1}
Given any term of the second line of any $\pi_c(\sigma)$, with $c>1$,
there exists a term of the second line of $\pi_{c-1}(\sigma)$ which is
larger and appears further to the left in $\sigma$.
\end{lem}

\begin{proof}
If not, then the term in question would have already been included in
the second line of $\pi_{c-1}(\sigma)$.
\end{proof}

In the combinatorics literature, it is customary to write permutations
in one-line notation, as sequences $(\sigma_1,\dots, \sigma_r)$ with
$\sigma_i = \sigma(i)$ for $i \in [r]$. Such sequences are often
called \emph{words}. Increasing and decreasing subsequences of a
permutation are always interpreted with respect to the one-line
notation (equivalently, the second line of the two-line notation).

\begin{prop}\label{p:1}
The length $k$ of the canonical factorisation of a permutation
$\sigma$ in $\Sym_r$ is always equal to the length of a longest
decreasing subsequence of~$\sigma$.
\end{prop}

\begin{proof}
From Lemma \ref{lem:1} it follows that there is a decreasing
subsequence with at least one term taken from each of the second lines
of $\pi_1(\sigma), \dots, \pi_k(\sigma)$.  This is necessarily of
maximal length, as no decreasing subsequence can be formed using two
terms from the second line of the same $\pi_c(\sigma)$.
\end{proof}


\begin{lem}\label{lem:rearranging}
  If the terms of any increasing subsequence of a permutation are
  permuted in a non-trivial way, the Coxeter length of the permutation
  strictly increases.
\end{lem}

\begin{proof}
This follows from the fact that the Coxeter length of a permutation is
equal to the number of its inversions.  Since any increasing
subsequence is contained in a maximal one and a permutation of the
original sequence is also a permutation of any maximal one in which it
is contained, it suffices to prove the result for maximal increasing
subsequences. So from now on we assume that our increasing subsequence
is maximal.

Next, we observe that it suffices to prove the result for a swap of
consecutive terms of the given subsequence, as any permutation of the
subsequence can be achieved by a succession of such swaps.

Recall that an inversion in a permutation $\sigma$ is a pair
$(\sigma_i,\sigma_j)$ such that $i<j$ and $\sigma_i>\sigma_j$.
Furthermore, recall \cite{Knuth}*{\S5.1.1} that the \emph{inversion
table} of a permutation $\sigma \in \Sym_r$ is the sequence $(b_1,
\dots, b_r)$ such that $b_j$ is the number of $\sigma_i$ to the left
of $j$ which exceed $j$. In other words, $b_j$ is the number of
inversions with second term $j$. Then $\ell(\sigma) = b_1+\cdots +
b_r$, where $\ell(\sigma)$ is the Coxeter length of $\sigma$.

Finally, consider two consecutive terms $\sigma_i < \sigma_j$ (where
$i<j$) of our maximal increasing subsequence. Maximality means that
every intermediate term in $(\sigma_{i+1}, \dots, \sigma_{j-1})$ lies
outside of the closed interval $[\sigma_i,\sigma_j]$. Swapping
$\sigma_i$ and $\sigma_j$ increases the inversion number
$b_{\sigma_i}$ by $1+N$ and decreases the inversion number
$b_{\sigma_j}$ by $N$, where $N$ is the number of intermediate terms
exceeding $\sigma_i$. All other inversion numbers in the inversion
table remain unchanged, so the net effect is an increase of $1$ in the
total number of inversions. Thus, $\ell(w)$ is incremented by $1$, and
the proof is complete.
\end{proof}
 
\begin{example}
The following examples illustrate the proof of
Lemma~\ref{lem:rearranging}. Take $\sigma = (539746182)$ in $\Sym_9$.
Its inversion table is $b = (6,7,1,3,0,2,1,1,0)$ so $\ell(\sigma) =
21$.

(i) Consider the maximal increasing subsequence $(3468)$ in $\sigma$.
Swapping its first two terms produces the permutation $\sigma' =
(549736182)$, with inversion table $b' = (6,7,4,1,0,2,1,1,0)$. Notice
that $b_3$ increased by $3$ and $b_4$ decreased by $2$, so
$\ell(\sigma') = 22$.

(ii) Swapping the last two terms of $(3468)$ gives $\sigma'' =
(539748162)$. Its inversion table is $b'' = (6,7,1,3,0,3,1,1,0)$. This
time $b_6$ increased by $1$ and all other terms of $b$ remained the
same, so again $\ell(\sigma'') = 22$.
\end{example}

The \emph{weight} of a sequence $\ui = (i_1, i_2, \dots, i_r)$ in
$[k]^r$ is the composition $\mu = (\mu_1, \dots, \mu_k)$, where
$\mu_c$ counts the number of terms in $\ui$ equal to $c$, for each $c$
in $[k]$. It is clear by construction that each of the sequences
$\ui$, $\uj$ such that $\varphi'(\sigma) = (\ui,\uj)$ is of the same
weight $\mu$.

The following result gives the main properties of the map $\varphi'$.

\begin{prop}\label{p:2}
The map $\varphi': \Sym_r \to [r]^r \times [r]^r$ is injective.  Given any
$\sigma$ in $\Sym_r$, let $\varphi'(\sigma) = (\ui,\uj)$. Then $\sigma$ is
the unique permutation in $\Sym_r$ of minimal length such that $\ui
\sigma = \uj$.
\end{prop}

\begin{proof}
To prove the injectivity, let $(\ui,\uj)$ be a pair of multi-indices
in $[r]^r \times [r]^r$ of the same weight $\mu = (\mu_1,\dots,
\mu_k)$.  For each $c$ in $[k]$, let $p(c) = (x(c)/(y(c))$ be the pair
of increasing sequences defined by letting $x(c)$ (resp., $y(c)$) be
the positions in the sequence $\ui$ (resp., $\uj$) having value $c$,
in order. Then the configuration $(p(1) \mid p(2) \mid \cdots \mid
p(k))$ defines, up to column reordering, the original permutation
$\sigma$ that produced the pair $(\ui,\uj)$. Sending $(\ui,\uj)
\mapsto \sigma$ defines a left inverse to $\varphi'$, and thus
$\varphi'$ is injective.  For instance, if
\[
(\ui,\uj) =  ((1,2,3,1,2,1,1,2,3), (3,2,2,1,3,2,1,1,1))
\]
then we have
\[
(p(1)\mid p(2) \mid p(3)) =
\left(\begin{array}{cccc|ccc|cc}1 & 4 & 6 & 7 & 2 & 5 & 8 & 3 & 9 \\
 4 & 7 & 8 & 9 & 2 & 3 & 6 & 1 & 5\end{array}\right)
\]
which recovers the permutation $\sigma$ of Example~\ref{e:1}.

Now we prove the second claim. Suppose that $\ui \sigma' = \uj$ for
some $\sigma' \ne \sigma$ in $\Sym_r$. Then $\sigma'$ is obtainable
from $\sigma$ by permuting the first (or, equivalently, second) line
of one or more parts of its canonical factorisation. If more than one
part is involved, they can be permuted independently, so it suffices
to consider such a permutation of a single part of the canonical
factorisation. This amounts to a reordering of the terms of an
increasing subsequence of $\sigma$, which by
Lemma~\ref{lem:rearranging} necessarily increases the Coxeter length.
\end{proof}

Now we can define the map $\varphi$. If $n \ge r$, we have $\B =
\Sym_r$ since $\Phi$ is faithful.  If $n < r$, $\B = \{\sigma \in
\Sym_r \mid \LLDS(\sigma) \le n\}$ is a proper subset of $\Sym_r$, and
we define
\[
\varphi = \text{ the restriction of $\varphi'$ to $\B$} .
\]
If $n \ge r$ then we define $\varphi = \varphi'$. The properties of
Proposition \ref{p:2} are invariant under restriction, thus are
applicable to $\varphi$ in all cases.

\begin{rmk}\label{rmk:1}
In the stable case $n \ge r$, we can do better than the unitriangular
subsystem provided by $\varphi = \varphi'$. Namely, for each $\sigma$,
simply set $\ui=(1,2,\dots,r)$ and $\uj=\ui \sigma$. Since $n \ge r$,
the values in $\uj$ are all distinct, so $\sigma$ is the unique
element of $\Sym_r$ such that $\ui\sigma = \uj$. Sending $\sigma$
to $(\ui,\uj)$ instead of using $\varphi = \varphi'$ picks out a
subsystem of \eqref{e:4} having a diagonal coefficient matrix $A$. In
fact, as it is a $(0,1)$-matrix, $A = I$ is the identity matrix.
\end{rmk}

\begin{example}\label{ex:r3}
(i) Suppose that $r=3$. The pairs $(\ui,\uj)$ corresponding to each
  $\sigma$ in $\Sym_3$ under $\varphi'$ are tabulated below:
\[
\begin{array}{c|cccccc}
\sigma & (123)&(213)&(132)&(231)&(312)&(321)\\ \hline
\ui &    (1,1,1)&(1,2,1)&(1,1,2)&(1,1,2)&(1,2,2)&(1,2,3)\\
\uj &    (1,1,1)&(2,1,1)&(1,2,1)&(2,1,1)&(2,2,1)&(3,2,1)
\end{array}
\]
Order the elements of $\Sym_3$ as indicated in the table; this is a
linear order compatible with Coxeter length.  The map $\varphi'$
picks out the subsystem of \eqref{e:4} given by
\begin{eqnarray*}
  c(123)+c(213)+\cdots+c(321) &=& X_{(1,1,1),(1,1,1)} \\
  c(213)+c(312) &=& X_{(1,2,1),(2,1,1)} \\
  c(132)+c(312) &=& X_{(1,1,2),(1,2,1)} \\
  c(231)+c(321) &=& X_{(1,1,2),(2,1,1)} \\
  c(312)+c(321) &=& X_{(1,2,2),(2,2,1)} \\
  c(321) &=&  X_{(1,2,3),(3,2,1)}
\end{eqnarray*}
with coefficient matrix $A$ of the form
\[
A =
\begin{bmatrix}
  1&1&1&1&1&1\\
  &1&\cdot&\cdot&1&\cdot\\
  &&1&\cdot&1&\cdot\\
  &&&1&\cdot&1\\
  &&&&1&1\\
  &&&&&1
\end{bmatrix}.
\]
By inverting the matrix $A$, we conclude that, for any $X$ in
$\End_{S(n,3)}(\V^{\otimes r})$, with $n \ge 3$, we have
\begin{align*}
  c(123) &= X_{(1,1,1),(1,1,1)} - X_{(1,2,1),(2,1,1)} - X_{(1,1,2),(1,2,1)}\\
  &\phantom{=} - X_{(1,1,2),(2,1,1)} + X_{(1,2,2),(2,2,1)} - X_{(1,2,3),(3,2,1)} \\
  c(213) &= X_{(1,2,1),(2,1,1)} - X_{(1,2,2),(2,2,1)} + X_{(1,2,3),(3,2,1)} \\
  c(132) &= X_{(1,1,2),(1,2,1)} - X_{(1,2,2),(2,2,1)} + X_{(1,2,3),(3,2,1)} \\
  c(231) &= X_{(1,1,2),(2,1,1)} - X_{(1,2,3),(3,2,1)} \\
  c(312) &= X_{(1,2,2),(2,2,1)} - X_{(1,2,3),(3,2,1)} \\
  c(321) &= X_{(1,2,3),(3,2,1)} .
\end{align*}

(ii) If $(n,r) = (2,3)$ then restricting the above to $\varphi$
picks out the subsystem of \eqref{e:4} given by
\begin{eqnarray*}
  c(123)+c(213)+\cdots+c(312) &=& X_{(1,1,1),(1,1,1)} \\
  c(213)+c(312) &=& X_{(1,2,1),(2,1,1)} \\
  c(132)+c(312) &=& X_{(1,1,2),(1,2,1)} \\
  c(231) &=& X_{(1,1,2),(2,1,1)} \\
  c(312) &=& X_{(1,2,2),(2,2,1)} 
\end{eqnarray*}
with coefficient matrix $A$ of the form
\[
A =
\begin{bmatrix}
  1&1&1&1&1\\
  &1&\cdot&\cdot&1\\
  &&1&\cdot&1\\
  &&&1&\cdot\\
  &&&&1
\end{bmatrix}.
\]
Notice that $A$ is a submatrix of the matrix in part (i).  By
inverting $A$, we conclude that, for any $X$ in
$\End_{S(2,3)}(\V^{\otimes r})$,
\begin{align*}
  c(123) &= X_{(1,1,1),(1,1,1)} - X_{(1,2,1),(2,1,1)} - X_{(1,1,2),(1,2,1)}\\
  &\phantom{=} - X_{(1,1,2),(2,1,1)} + X_{(1,2,2),(2,2,1)}  \\
  c(213) &= X_{(1,2,1),(2,1,1)} - X_{(1,2,2),(2,2,1)}  \\
  c(132) &= X_{(1,1,2),(1,2,1)} - X_{(1,2,2),(2,2,1)}  \\
  c(231) &= X_{(1,1,2),(2,1,1)}  \\
  c(312) &= X_{(1,2,2),(2,2,1)} .
\end{align*}
\end{example}

\begin{rmk}
In general, if $n<r$ then we can always obtain a solution from one in
the stable case $n \ge r$ (we may as well take $n=r$) by simply
putting $X_{\ui\sigma, \uj\sigma} = 0$ for all $\sigma$ not in $\B =
\B_{n,r}$. This is a consequence of the fact that $\varphi$ is defined
to be the restriction of $\varphi'$.  Example~\ref{ex:r3} illustrates
this phenomenon.
\end{rmk}

\section{\bf The map $\psi$}\label{s:4}\noindent
We continue to assume that $\Bbbk$ is a unital commutative ring. Now
we consider the other symmetric group $W_n$. We write $\C = \C_{n,r}$,
$A = A_\Psi$ in this section. Our goal is to define the combinatorial
map $\psi$ and show that it has the properties needed to prove
Theorem~\ref{t:3}(ii).

The defining condition in the set $\C = \{w \in W_n \mid \LLIS(w) \ge
n-r \}$ is vacuous if $n-r \le 1$, equivalently, if $r \ge n-1$.  The
condition $r \ge n-1$ determines the stable case in this situation.
Note that (by e.g., \cites{BDM:kernel,BDM}) the representation $\Psi:
\Bbbk[W_n] \to \End_{\Ptn_r(n)}(\V^{\otimes r})$ is faithful if and
only if $r \ge n-1$, in which case $\C = W_n$.

\begin{defn}
Unlike $\varphi$, which depended only on $r$, the map $\psi$
always depends on both $n$, $r$.
\begin{enumerate}
\item Suppose that $r \le n-1$. If $w = (w_1, w_2, \dots, w_n)$ is in
  $\C$, where $w_j = w(j)$ for all $j$ in $[n]$, we choose (any way we
  like) an increasing subsequence of $w$ of length $n-r$, and let
  \[
  (w_{j_1}, w_{j_2}, \dots, w_{j_r}), \quad j_1 < j_2 < \cdots < j_r
  \]
  be its \emph{complement} in $w$.  Set $i_\alpha = w_{j_\alpha} =
  w(j_\alpha)$ for each $\alpha = 1, \dots, r$. If $w$ is written in
  two-line notation then $(i_1,\dots,i_r)$ is the complement of the
  chosen increasing subsequence in the second line and $(j_1,\dots,
  j_r)$ is the corresponding subsequence in the first line. We then
  define $\psi(w) = (\ui,\uj)$ where $\ui = (i_1,\dots,i_r)$ and $\uj
  = (j_1,\dots, j_r)$.

\item Notice that $\C = W_n$ if $r=n-1$, and the map $\psi$ sends
  $W_n$ into $[n]^{n-1} \times [n]^{n-1}$. If $r>n-1$ we extend this
  to a map $W_n \to [n]^{r} \times [n]^{r}$ by appending an additional
  $r-(n-1)$ values of $1$, $w_1=w(1)$ to the end of each $\uj$, $\ui$
  respectively.  The choice of $1$ here is purely arbitrary, and it
  can be replaced by any other number in $[n]$.
\end{enumerate}
\end{defn}

\begin{example}
Suppose that $w = (421738965)$ in $W_9$. Then we have $\LLIS(w) =
4$. Since the condition $\LLIS(w) \ge n-r$ is equivalent to $r \ge n -
\LLIS(w)$, we may take any $r \ge 5$, up to and including $r=8$.
\begin{enumerate}\renewcommand{\labelenumi}{(\roman{enumi})}
\item The choice $r=5$ with the increasing subsequence $(2389)$
  yields the pair $(\ui,\uj) = ((4,1,7,6,5),(1,3,4,8,9))$.

\item The choice $r=6$ with the increasing subsequence $(238)$
  yields the pair $(\ui,\uj) = ((4,1,7,9,6,5),(1,3,4,7,8,9))$.
\end{enumerate}
The pair $(\ui,\uj)$ clearly depends on the chosen increasing
subsequence of $w$. Notice that $\uj$ is necessarily increasing, and
that is always the case.
\end{example}

The following result completes the proof of Theorem~\ref{t:3}.

\begin{prop}\label{p:3}
The map $\psi: \C \to [n]^r \times [n]^r$ is injective. Given any $w$
in $W_n$, let $\psi(w) = (\ui,\uj)$. Then $w$ is the unique permutation
in $W_n$ of minimal length such that $w \uj  = \ui$.
\end{prop}

\begin{proof}
To prove the injectivity claim, we need to show that for any given
pair $(\ui,\uj)$ in the image of $\psi$, there is precisely one $w \in
\C$ for which $\psi(w) = (\ui,\uj)$.  In the stable case $r \ge n-1$,
this is clear because any permutation of $n$ items is determined once
its values on any $n-1$ of them is specified.  If $r < n-1$ this still
holds, because the given pair $(\ui,\uj)$ determines $w$ on the
numbers in $\uj$, and $w$ is determined on all other numbers in $[n]$
by the fact that it must be increasing on those numbers. Hence $\psi$
is injective.

Now consider the second claim. It is vacuous in the stable range $r
\ge n-1$, in which case there is only one $w$ such that $w \uj = \ui$.
If $ r < n-1$, let $w' \ne w$ be any permutation in $W_n$ such that
$w' \uj = \ui$.  Then $w'$, $w$ necessarily agree on $\{j_1, j_2,
\cdots, j_r\}$. This means that $w'$ may be obtained from $w$ by
permuting its chosen increasing subsequence of length $n-r$, indexed
by the complement $[n] \setminus \{j_1, \dots, j_r\}$. Such a
permutation necessarily increases the total number of inversions, by
Lemma~\ref{lem:rearranging}, and thus the Coxeter length.
\end{proof}

\begin{example}
(i) Take $(n,r)=(3,2)$. If we always choose the first term of $w$ as
  the increasing subsequence $s$ of length $n-r=1$, the pairs
  $(\ui,\uj)$ such that $w \uj = \ui$ corresponding to each $w$ in
  $\C=W_3$ under $\psi$ are:
\[
\begin{array}{c|cccccc}
  w & (123)&(213)&(132)&(231)&(312)&(321)\\ \hline
  s & (1)&(2)&(1)&(2)&(3)&(3) \\
  \ui & (2,3)& (1,3) & (3,2) & (3,1) & (1,2) & (2,1) \\
  \uj & (2,3)& (2,3) & (2,3) & (2,3) & (2,3) & (2,3)
\end{array}
\]
Order the elements of $W_3$ in the reading order across the table.
This is a linear order compatible with Coxeter length.  Then the
matrix $A$ produced by $\varphi$ is just $A = I$, the identity matrix.

(ii) Take $(n,r)=(3,1)$. We choose $s$ to be the first increasing
subsequence of $w$ of length $n-r=2$. Then the pairs $(\ui,\uj)$ such
that $w \uj = \ui$ corresponding to each $w$ in $\C$ under $\psi$ are:
\[
\begin{array}{c|ccccc}
  w & (123)&(213)&(132)&(231)&(312)\\ \hline
  s & (12) & (23) & (13) & (23) & (12) \\
\ui &  3   &   1   &   2   &   1   &   3  \\
\uj &  3   &   2   &   3   &   3   &   1  
\end{array}
\]
We again order the elements of $\C$ by the reading order across the
table.  The subsystem of \eqref{e:4} picked out by $\psi$ is
\begin{eqnarray*}
  c(123)+c(213) &=& X_{3,3} \\
  c(213)+c(312) &=& X_{1,2} \\
  c(132)+c(312) &=& X_{2,3} \\
  c(231) &=& X_{1,3} \\
  c(312) &=& X_{3,1}
\end{eqnarray*}
with coefficient matrix 
\[
A = 
\begin{bmatrix}
  1&1&\cdot&\cdot&\cdot\\
  &1&\cdot&\cdot&1\\
  &&1&\cdot&1\\
  &&&1&\cdot\\
  &&&&1
\end{bmatrix} .
\]
Inverting $A$ yields the equations
\begin{alignat*}{3}
  c(123) &= X_{3,3} - X_{1,2} + X_{3,1} \\
  c(213) &= X_{1,2} - X_{3,1} \quad
 &\qquad c(132) &= X_{3,3}- X_{3,1} \\
  c(231) &= X_{1,3} \quad
 & c(312) &= X_{3,1} .
\end{alignat*}
\end{example}

\section{\bf Further remarks}\label{s:5}\noindent
We continue to assume that $\Bbbk$ is a commutative ring with $1$. Fix
$n$ and $r$ and let $\Phi = \Phi_{n,r}$ and $\Psi = \Psi_{n,r}$.

Let $\tau \in \Sym_r$ and $y \in W_n$. If we take $X = \Phi(\tau)$ or
$\Psi(y)$ in the algebra $\mathbf{B}(n,r) = \Bbbk[\Sym_r]^\opp/\ker(\Phi)$
or $\mathbf{C}(n,r) = \Bbbk[W_n]/\ker(\Psi)$ then we have an algorithm
that computes the coefficents $c(\sigma)$ for all $\sigma \in
\B_{n,r}$ or $c(w)$ for all $w \in \C_{n,r}$ in equation
\eqref{e:problem}. By running the algorithm with the indicated value
of $X$ we obtain a relation of the form
\begin{equation}\label{e:normal-form}
  \bar{\tau} = \sum_{\sigma \in \B_{n,r}} c(\sigma)\bar{\sigma}
  \quad \text{or} \quad
  \bar{y} = \sum_{w \in \C_{n,r}} c(w)\bar{w}.
\end{equation}
We call such a linear combination the \emph{normal form} of
$\bar{\tau}$ or $\bar{y}$ with respect to the basis $\bar{\B}_{n,r}$
or $\bar{\C}_{n,r}$, respectively. Obviously, the above relation is
non-trivial if and only if $\bar{\tau}$ is not in $\bar{\B}_{n,r}$ or
$\bar{y}$ is not in $\bar{\C}_{n,r}$. The following is clear.

\begin{prop}
The algebra $\mathbf{B}(n,r)$ is isomorphic to the algebra generated
by the $\bar{\sigma}$ for $\sigma \in \B_{n,r}$ subject to the
relations given by multiplication in $\Sym_r$ along with the normal
form equations in \eqref{e:normal-form} imposed on all products of
elements of $\B_{n,r}$.  The algebra $\mathbf{C}(n,r)$ is isomorphic
to the algebra generated by the $\bar{w}$ for $w \in \C_{n,r}$ subject
to the relations given by multiplication in $W_n$ along with the
normal form equations in \eqref{e:normal-form} imposed on all products
of elements of $\C_{n,r}$.
\end{prop}

These presentations are in general far from being minimal. They are of
course equivalent to the presentations obtained by imposing the group
multiplication on the ambient symmetric group along with any set of
relations generating $\ker(\Phi)$ or $\ker(\Psi)$.

\begin{example}
In these examples, we express normal form equations as congruences
modulo the appropriate kernel in the corresponding quotient of the
group algebra.

(i) For $\mathbf{B}(1,r)$ or $\mathbf{C}(n,0)$ the normal form of
every permutation is just the identity. That is, every permutation is
congruent to the identity.

(ii) For $\mathbf{B}(r-1,r)$ or $\mathbf{C}(n,n-2)$ the normal form of
the permutation of longest Coxeter length is given by the alternating
sum of all the the other permutations. For instance, in
$\mathbf{C}(n,n-2)$ we always have the normal form congruence:
$w_0 \equiv \sum_{w\ne w_0} (-1)^{\ell(w)+1} w$. All other
normal forms are trivial. The analogous formula holds in the other
case. This is clear, as $\sum_w (-1)^{\ell(w)} w$ generates the
kernel of $\Psi_{n,n-2}$ and similarly in the other case.

(iii) This leaves $\mathbf{B}(2,4)$ as the sole remaining case of
interest when $r=4$.  See Example~\ref{ex:bases} for a list of the
elements of $\B_{2,4}$.  In that case we have the following
non-trivial normal form congruences, where as usual permutations are
written in one-line notation:
\begin{align*}
  (4321) &\equiv -(1234)+(1243)+2 (1324)-(1342)-(1423)+(2134)\\
         &  -(2143)-(2314)+(2341)-(3124)+(3412)+(4123).\\
  (4312) &\equiv -(1234)+(1243)+(1324)-(1423)+(2134)-(2143)\\
         &  -(3124)+(3412)+(4123).\\
  (4231) &\equiv (1243)+(1324)-(1342)-(1423)+(2134)-2 (2143)\\
         & -(2314)+(2341)+(2413)-(3124)+(3142)+(4123).\\
  (3421) &\equiv -(1234)+(1243)+(1324)-(1342)+(2134)-(2143)\\
         & -(2314)+(2341)+(3412).\\
(4213) &\equiv (1243)-(1423)-(2143)+(2413)+(4123).\\
(4132) &\equiv (2134)-(2143)-(3124)+(3142)+(4123).\\
(3241) &\equiv (1243)-(1342)-(2143)+(2341)+(3142).\\
(2431) &\equiv (2134)-(2143)-(2314)+(2341)+(2413).\\
(3214) &\equiv (1234)-(1324)-(2134)+(2314)+(3124).\\
(1432) &\equiv (1234)-(1243)-(1324)+(1342)+(1423).
\end{align*}

(iv) In case $n=4$, the sole interesting case is $\mathbf{C}(4,1)$. In
that case the non-trivial normal form congruences are as follows:
\begin{align*}
(4321) &\equiv (1234)-(1243)-(2134)+(2341)+(4123).\\
(4312) &\equiv (1234)-(1243)-(1324)+(1342)-(2134)+(2314)+(4123).\\
(4231) &\equiv 2 (1234)-(1243)-(1324)-(2134)+(2341)+(4123).\\
(3421) &\equiv (1234)-(1243)-(1324)+(1423)-(2134)+(2341)+(3124).\\
(4213) &\equiv (1234)-(1324)-(2134)+(2314)+(4123).\\
(4132) &\equiv (1234)-(1243)-(1324)+(1342)+(4123).\\
(3412) &\equiv (1234)-(1243)-2 (1324)+(1342)+(1423)\\&-(2134)+(2314)+(3124).\\
(3241) &\equiv (1234)-(1324)-(2134)+(2341)+(3124).\\
(2431) &\equiv (1234)-(1243)-(1324)+(1423)+(2341).\\
(3214) &\equiv (1234)-(1324)-(2134)+(2314)+(3124).\\
(3142) &\equiv -(1324)+(1342)+(3124).\\
(2413) &\equiv -(1324)+(1423)+(2314).\\
(1432) &\equiv (1234)-(1243)-(1324)+(1342)+(1423).\\
(2143) &\equiv -(1234)+(1243)+(2134).
\end{align*}
This is the end of the examples.
\end{example}

The normal form congruences determine the structure constants for each
algebra with respect to its canonical basis.

\begin{prop}
Let $b_{\sigma,\tau}^\pi$ be the structure constants of the algebra
$\mathbf{B}(n,r)$ with respect to its canonical basis, defined by the
equations
\[
\overline{\sigma\tau} =  \bar{\sigma}\bar{\tau} = \sum_{\pi \in
  \B_{n,r}} b_{\sigma,\tau}^\pi\, \bar{\pi}
\]
for any $\sigma,\tau$ in $\B_{n,r}$. The right hand side of the above
is given by the normal form of the product $\sigma\tau$ in $\Sym_r$,
hence the structure constant $b_{\sigma,\tau}^\pi$ depends only on
$\sigma\tau$ and $\pi$. The analogous statements apply to the
structure constants of $\mathbf{C}(n,r)$ with respect to its canonical
basis.
\end{prop}

Recall that $\Sym_r$ is a subgroup of the multiplicative monoid
$\Ptn_r(n)^\times$ of invertible elements in the partition algebra
$\Ptn_r(n)$. This immediately implies that
\begin{equation}
\mathbf{C}(n,r) = \End_{\Ptn_r(n)}(\V^{\otimes r})
\subset \End_{\Sym_r}(\V^{\otimes r}) = S(n,r).
\end{equation}
In other words, the quotient algebra $\Bbbk[W_n]/\ker(\Psi_{n,r})$ is
isomorphic to a subalgebra of the Schur algebra.  

Recall that $\Lambda(n,r)$ is defined in \cite{Green}*{\S3.1} as the
set of orbits of the action of $\Sym_r$ on $I(n,r) = [n]^r$. Elements
of $\Lambda(n,r)$ are called \emph{weights}; we already encountered
this notion in the paragraph following Lemma~\ref{lem:rearranging}. A
weight $\alpha$ is specified by the vector $\alpha = (\alpha_1, \dots,
\alpha_n)$ such that $\alpha_j$ counts the number of times that $j$
appears in $\ui = (i_1, \dots, i_r)$, for any $\ui \in I(n,r)$, where
$j$ varies over $[n]$. The action of $W_n$ on $\Lambda(n,r)$ is given
by
\begin{equation}
w^{-1} \alpha = (\alpha_{w(1)}, \dots, \alpha_{w(n)}).
\end{equation}
The following result answers a question of R.~Rouquier.

\begin{prop}
In terms of Green's basis $\xi_{\ui,\uj}$ of the Schur algebra
$S(n,r)$, the canonical basis elements $\bar{w}$ for $w \in \C_{n,r}$
satisfy the formula
\[
\bar{w} = \textstyle \sum_{\ui} \xi_{\ui,w(\ui)} 
\]
where the sum is carried out over any set $\mathcal{O}(n,r)$ of orbit
representatives of $\Lambda(n,r)$. 
\end{prop}

\begin{proof}
By \cite{Green}*{\S2.6}, the matrix of $\xi_{\ui,\uj}$ with respect to
the basis $\vv_{\uj} = \vv_{j_1} \otimes \cdots \otimes \vv_{j_r}$ is
given by the $n^r \times n^r$ matrix $T(\omega) =
[T(\omega)_{\up,\uq}]_{\up,\uq}$ defined by
\[
T(\omega)_{\up,\uq} =
\begin{cases}
  1 & \text{ if } (\up,\uq) \sim (\ui,\uj) \\
  0 & \text{ otherwise,}
\end{cases}
\]
where $\omega \in \Omega(n,r)$ is the $\Sym_r$-orbit containing the
pair $(\ui,\uj)$. On the other hand, the matrix of $\Psi(w)$ with
respect to the same basis (for any $w \in W_n$) is the Kronecker power
$P(w)^{\otimes r}$, where $P(w) = [\delta_{i,w(j)}]_{i,j \in [n]}$ is
the permutation matrix representing $w$ with respect to the basis
$\{\vv_1,\dots, \vv_n\}$.  Thus, writing $w(\uj) = (w(j_1), \dots,
w(j_r))$ for any $\uj \in [n]^r$, we have
\[
\Psi(w) = P(w)^{\otimes r} = [\delta_{\ui, w(\uj)}]_{\ui,\uj \in
  [n]^r}.
\]
By expressing the matrices of $\Psi(w)$ and $\xi_{\ui,\uj}$ in terms
of the matrix units, one deduces the result.
\end{proof}

\subsection*{Acknowledgements}
The authors are grateful to the anonymous referee for a number of
insightful suggestions for improvement.  The first author was funded
by EPSRC fellowship grant EP/V00090X/1.  The second and third authors
thank the Isaac Newton Institute for Mathematical Sciences, Cambridge,
for providing support during the 2022 programme \emph{Groups,
representations and applications: new perspectives}, under the EPSRC
grant EP/K032208/1.


\end{document}